\documentclass{scrartcl}
\usepackage[utf8]{inputenc}
\usepackage[T1]{fontenc}
\usepackage{amsmath,amssymb,amsthm}
\usepackage[a4paper]{geometry}
\usepackage{graphicx}
\usepackage{xcolor}
\usepackage{cancel}
\usepackage{soul}

\usepackage{newunicodechar}  
\newunicodechar{α}{\alpha}
\newunicodechar{β}{\beta}
\newunicodechar{γ}{\gamma}
\newunicodechar{Γ}{\Gamma}
\newunicodechar{κ}{\kappa}
\newunicodechar{μ}{\mu}
\newunicodechar{ℝ}{\mathbb{R}}
\newunicodechar{Δ}{\Delta}
\newunicodechar{λ}{\lambda}
\newunicodechar{Λ}{\Lambda}
\newunicodechar{ν}{\nu}
\newunicodechar{∞}{\infty}
\newunicodechar{φ}{\varphi}
\newunicodechar{ξ}{\xi}
\newunicodechar{δ}{\delta}
\newunicodechar{Φ}{\Phi}
\newunicodechar{σ}{\sigma}
\newunicodechar{ζ}{\zeta}
\newunicodechar{ℕ}{\mathbb{N}}
\newunicodechar{ε}{\varepsilon}
\newunicodechar{π}{\pi}
\newunicodechar{∂}{\partial}

\newcommand{\f}[2]{\frac{#1}{#2}}
\newcommand{\ctilde}{\widetilde{c}}
\newcommand{\Om}{\Omega}
\newcommand{\Ombar}{\overline{\Omega}}
\newcommand{\Omtilde}{\widetilde{\Om}}
\newcommand{\dOm}{\partial\Om}
\newcommand{\delny}{\partial_{\nu}\ }
\newcommand{\nn}{\nonumber}
\newcommand{\norm}[2][]{\|#2\|_{#1}}
\newcommand{\Lp}[2][\Omega]{L^{#2}(#1)}
\newcommand{\kl}[1]{\left( #1 \right)}
\newcommand{\sub}{\subset}
\newcommand{\set}[1]{\left\{#1\right\}}
\newcommand{\na}{\nabla}
\newcommand{\io}{\int_\Omega}
\newcommand{\wtilde}{\widetilde{w}}
\newcommand{\folge}[2][n]{(#2)_{#1\in ℕ}}

\newtheorem{theorem}{Theorem}[section]
\newtheorem{lemma}[theorem]{Lemma}
\newtheorem{proposition}[theorem]{Proposition}
\newtheorem{cor}[theorem]{Corollary}

\title{Stationary solutions to a chemotaxis--consumption model with realistic boundary conditions}

\author{Marcel Braukhoff\footnote{Institute for Analysis and Scientific Computing, TU Wien, Wiedner Hauptstr. 8-10, 1040 Wien, Austria; e-mail: braukhoff@posteo.de}  \and 
  Johannes Lankeit\footnote{
  Department of Applied Mathematics and Statistics, Comenius University, Mlynsk\'a dolina, 84248 Bratislava, Slovakia; e-mail: jlankeit@math.uni-paderborn.de}\ \footnote{Institut für Mathematik, Universität Paderborn, Warburger Str. 100, 33098 Paderborn, Germany}
}

\begin{document}
\maketitle 
\begin{abstract}
\setlength{\parindent 0pt}\setlength{\parskip 2pt} \noindent \small 
 Previous studies of chemotaxis models with consumption of the chemoattractant (with or without fluid) have not been successful in explaining pattern formation even in the simplest form of concentration near the boundary, which had been experimentally observed.

 Following the suggestions that the main reason for that is usage of inappropriate boundary conditions, 
 in this article we study solutions to the stationary chemotaxis system 
\[
 \begin{cases}
  0 = \Delta n - \nabla\cdot(n\nabla c)\\
  0 = \Delta c - nc 
 \end{cases}
\]
in bounded domains $\Omega\subset\mathbb{R}^N$, $N\ge 1$, under no-flux boundary conditions for $n$ and the physically meaningful condition 
\[
 \partial_{\nu} c =  (\gamma-c)g
\]
on $c$, with given parameter $\gamma>0$ and $g\in C^{1+\beta}(\Omega)$ satisfying $g\ge 0$, $g \not\equiv 0$ on $\partial \Om$. We prove existence and uniqueness of solutions for any given mass $\int_\Omega n > 0$. These solutions are non-constant.\\
{\footnotesize 
}
\end{abstract}

\section{Introduction}

\subsection{Chemotaxis--consumption models}
Chemotaxis models with signal consumption, like 
\begin{equation}\label{eq:ctconsumption-parabolic}
 \begin{cases}
  n_t = \Delta n - \nabla\cdot(n\nabla c),\\
  c_t = \Delta c - nc,
 \end{cases}
\end{equation}
have received quite some interest over the last decade, especially in the context of chemotaxis--fluid models that had been introduced in  \cite{tuval}, see, e.g., \cite{tao_bdoxygenconsumption}, Sections 4.1 and 4.2 of the survey \cite{BBTW} or the introduction of \cite{cao_lankeit} and references therein. 

Here, $n$ denotes the concentration of some bacteria (for example of the species \textit{Bacillus subtilis}), whose otherwise random motion is partially directed towards higher concentrations $c$ of a signalling substance (oxygen) they consume.  

Accounting for a liquid environment, these equations are then usually coupled to incompressible Navier--Stokes or Stokes equations with a driving force $-n\nabla \Phi$ arising from density differences between fluid containing large or small amounts of bacteria:

\begin{equation}\label{eq:ctns}
 \begin{cases}
  n_t + u\cdot \nabla n= \Delta n - \nabla\cdot(n\nabla c),\\
  c_t + u \cdot \nabla c= \Delta c - nc,\\
  u_t + u\nabla u = \Delta u +\nabla P - n\nabla \Phi, \qquad \nabla \cdot u= 0.
 \end{cases}
\end{equation}
 
One of the main questions motivating the study of this system and its relatives was whether and how \eqref{eq:ctns} can account for the emergence of large scale coherent patterns, as observed experimentally in \cite{dombrowskietal,tuval}. 
There were also other motivations; see, e.g., the question posed in the title of \cite{win_transAMS}, or 
the introduction of \cite{blm3} and references therein; 
but at least with regard to the first-mentioned matter, 
results on the long-term behaviour of solutions to \eqref{eq:ctns}, paint a different picture: 
 
Not only small-data solutions to \eqref{eq:ctns} in three-dimensional domains (see e.g. \cite{cao_lankeit} or \cite{chae_kang_lee,xia_ye,tan_zhou}), but also every classical solution in $\Om\subset ℝ^2$  (\cite{win_fluid_konvergenzresultat,zhang_li_decay,jiang_wu_zheng,fan_zhao}) and even every ``eventual energy solution'' to \eqref{eq:ctns} converges to the stationary, constant state $(\f1{|\Om|}\io n_0,0,0)$, \cite{win_transAMS}. 

Also if the diffusion in the first equation of \eqref{eq:ctns} is of porous-medium type (see e.g. \cite{difrancesco_lorz_markowich}) and the chemotaxis term is of a more general form (\cite{win_CalcVarPDE}), solutions tend to a constant equilibrium. (Analogues for the fluid-free settings exist: \cite{tao_win,fan_jin,lswx}.)


Even the combination with logistic population growth terms ($+κn-μn^2$), whose interplay with chemotaxis systems of signal production type 
is known to result in quite colourful and nontrivial dynamics (\cite{hillenpainter_spatiotemporalchaos,win_transient,lankeit_thresholds}), 
does nothing to change these circumstances: 
In \cite{lankeit_m3as}, weak solutions have been constructed that eventually become smooth -- and converge to the spatially homogeneous state $(\f{κ}{μ},0,0)$. Also in the fluid-free setting every bounded solution converges to the constant state, \cite{lankeit_wang}. This trend towards homogeneous equilibria moreover extends to scenarios of food-supported proliferation, \cite{win_nutrienttaxis}.

Yet, apparently, convergence to a constant, and hence structureless, state suggests the opposite of pattern formation.---Nevertheless, it might be possible that interesting dynamics occur within a smaller timeframe (cf. \cite{win_transient,lankeit_thresholds} for a corresponding observation in signal-production chemotaxis systems; and even finite-time blow-up has not been excluded (but neither proven) for some settings); 
there is, however, 
another possible 
culprit for this strong discrepancy between experimental and theoretical outcomes  
that, in our opinion, should be investigated first: 

\subsection{The boundary condition}
In \cite[page 2279]{tuval}, Tuval et al.\ state that the ``boundary conditions [...] are central to the global flows and possible singularities''. 
All of the above-mentioned results use homogeneous Neumann boundary conditions for both $n$ and $c$, which may be mathematically convenient, but is not entirely realistic, for while it seems reasonable to assume that the bacteria do not leave or enter the domain (a drop of water), and oxygen does not penetrate the part of the boundary that is comprised of the area of contact between the drop and a surface on which it is resting, the interface between the fluid and surrounding air certainly does admit passage of oxygen, especially if its concentration in the water has plummeted due to activity of the bacteria. 

Instead we propose to prescribe the following boundary condition, a derivation of which can be found in 
\cite{braukhoff}: 
In accordance with Henry's law modelling the dissolution of gas in water, cf. \cite[sec. 5.3, p.144]{atkins}, we consider
\begin{equation}\label{our-bc-simplified}
 \delny c(x) = (γ-c(x))g(x)\qquad \text{for }x\in \partial \Om, 
\end{equation}
where the constant $γ>0$ denotes the maximal saturation of oxygen in the fluid and the influx of oxygen is proportional to the difference between current and maximal concentration on the fluid surface (see \cite[section 5.3 on page 144]{atkins}). The function $g$ models the absorption rate of the gaseous oxygen into the fluid. The gaseous oxygen concentration can assumed to be constant, because the oxygen-diffusion coefficient  in air is three orders of magnitude larger than that in the	fluid \cite[page 2279]{tuval}. 
This would lead to a constant absorption rate (i.e. $g=const.$) if all of the boundary were part of the water--air interface. Since this is not the case in general, we will incorporate a function accounting for different permeabilities of different parts of the boundary: no flux should take place on the boundary between water and a solid surface:
We consider $g=0$ on the solid--water interface and $g\ge 0$, but $g\not\equiv 0$, so that there is some water--air boundary. 

In \cite{tuval}, it is assumed that the absorption rate at the  water--air boundary is large such that the Dirichlet boundary conditions 
\[0=(c(x)-\gamma)\qquad \text{for }x\in \partial \Om\]
were posed, 
corresponding to a formal limit of $g\to\infty$ in \eqref{our-bc-simplified}. (Cf. also Proposition \ref{boundaryproposition} below.)

We will show 
that then the stationary system corresponding to \eqref{eq:ctconsumption-parabolic}, i.e. 
\begin{align}\label{system}
\begin{cases}
 0 = Δn- \nabla\cdot(n\nabla c)&\text{in } \Om, \\
 0= Δc-nc&\text{in } \Om,\\
 \delny c = (γ-c)g&\text{on } \dOm,\\
 \delny n = n\delny c&\text{on } \dOm
\end{cases}
\end{align}
has a solution, which moreover is unique for each prescribed total bacterial mass $\io n$ and is non-constant.

\subsection{Classical, signal-production chemotaxis systems}
This result offers a contrast to the classical Keller--Segel model 
\begin{equation}\label{classKS}
 \begin{cases}
  u_t= Δu - \nabla \cdot (u\nabla v)&\text{in } \Om,\\
  v_t= Δv - v +u &\text{in } \Om,\\
  \delny u=\delny v = 0 &\text{on } \dOm, 
 \end{cases}
\end{equation}
whose stationary solutions by the striking result of \cite{feireisl_laurencot_petzeltova} are known to serve as limit for global solutions, but form a much more complicated set, see \cite{senba_suzuki}; in particular solutions are non-unique: Constants obviously solve the stationary problem of \eqref{classKS}, but there are also non-constant solutions, see \cite[Sec. 6]{biler} for the radially symmetric case and \cite[Sec. 5]{horstmann_nonsymmetric}, \cite{wang-wei} as well as \cite{senba_suzuki} and \cite{kabeya_ni}. 

%
%

The situation for related systems, like Keller--Segel with logarithmic sensitivity, is similar, as studies of the ``Lin--Ni--Takagi problem'' show (see \cite{lin_ni_takagi} and its descendants). More on the question of nonhomogeneous stationary solutions and bifurcation analysis in a large class of Keller--Segel like systems (that is, with quite diverse parameters and possible nonlinearities, but always homogeneous Neumann boundary conditions) can also be found in the classical article \cite{schaaf} by R. Schaaf. 



\subsection{Previous work on chemotaxis--consumption models with other boundary conditions}
Signal-consumption models with boundary conditions different from homogeneous Neumann conditions have primarily appeared in numerical experiments that recover patterns like those experimentally observed, see \cite{tuval}, \cite{chertock_etal_numeric}, \cite{lee_kim_numerical_bioconvection}.

Analytical results that include such boundary conditions are the following: In \cite{Lorz}, the paper that started the mathematical study of chemotaxis-fluid systems and proves local existence of weak solutions, an inhomogeneous Dirichlet condition for $c$ on parts of the boundary of a bounded domain in $ℝ^2$ is mentioned; we have already referred to \cite{braukhoff}, where global weak solutions to a chemotaxis fluid model with logistic growth are proven to exist under \eqref{our-bc-simplified}. The recent work \cite{preprint_zhaoyin} deals with the domain $ℝ^2\times(0,1)$, imposing inhomogeneous Dirichlet condition $c=γ$ on the ``top'' surface and proving existence and convergence of solutions starting ($H^2$-)close to $(0,γ,0)$. (For technical reasons, the proof needs a consumption term that grows at least quadratically in $n$.) 
Non-zero boundary data, in form of either a Dirichlet or a Neumann condition, are also posed in \cite{knosalla_global}, where the system 
\[
 \begin{cases}
  n_t=n_{xx} - (n E(c)_x)_x,\\
  c_t=c_{xx} - n E(c)
 \end{cases}
\]
is studied in one-dimensional domains, for an ``energy function'' $E$ which has only one local maximum and satisfies $E(c)\to 0$ for both $c\to 0$ and $c\to \infty$. Existence of global, bounded solutions is shown. Steady states of this system have been considered in \cite{knosalla_nadzieja_stationary}. Their existence and uniqueness depends on the relation between the total mass $\io n$ and the size of the boundary data. 

Up to now, results on stationary states in higher dimension and any treatment of the boundary condition \eqref{our-bc-simplified} beyond existence of weak solutions to the parabolic problem are missing. This is the gap we intend to fill with the present article. 

%
%
%
%
%
%
%
%
\subsection{Statement of the main result and plan of the article }

In order to give the main theorem, let us first specify the more technical assumptions that we will make:  

With some numbers $N\in ℕ$ and $β_*\in(0,1)$, assumed to be fixed throughout the article, we will usually pose the following condition on the domain:
\begin{equation}\label{cond:Om}
 \Om\subset ℝ^{N} \text{ is a bounded domain with } C^{2+β_*}\text{-boundary}. 
\end{equation}

As motivated above, in stating the boundary condition we will use some function 

\begin{equation}\label{condition:g1}
 g\in C^{1+β_*}(\Ombar),\quad g \ge 0 \text{ on } \dOm, \quad
 g \not\equiv 0 \text{ on } \partial\Om.
\end{equation}

With these, we can state our main results: 

\begin{theorem}\label{thm:main}
 Let $\Om$ satisfy \eqref{cond:Om}, let $g$ be as in \eqref{condition:g1} and $γ>0$. 

 For every $m>0$ there is exactly one pair $(n,c)\in (C^{2}(\Om)\cap C^1(\Ombar))^2$ that solves \eqref{system} and satisfies $\io n = m$. This solution is positive in $\Ombar$ in both components, but not constant. 
 \end{theorem}
While the result that the solutions are non-constant is already well in line with the desired outcome, it seems expedient to attempt to gain further insight into the shape of solutions, at least in particular situations.

\begin{theorem}\label{thm:main2}
	Let $\Om=B_R(0)\subset ℝ^N$ be the ball for some $R>0$, let $g>0$ and $γ>0$ be constant and $(n,c)$ the solution of \eqref{system} by Theorem \ref{thm:main}. Then $n$ and $c$ are strictly convex.
\end{theorem}
Theorem \ref{thm:main} will be proven at the end of Section \ref{sec:system}. One of the keystones of this proof is the observation that \eqref{system} can be transformed into the scalar problem 
\begin{equation}\label{scalarbvp}
 \begin{cases} Δc=αce^c\qquad &\text{ in } \Om,\\
 \delny c = (  γ-c )g \qquad &\text{ on } \partial \Om,
\end{cases}
\end{equation}
for some parameter $α$, if $n=αe^c$, cf. also \cite[Thm. 2.1]{schaaf} for the classical Keller--Segel system. (The first equation, i.e. the equation this result is concerned with, is identical, although there is a miniscule difference in the boundary conditions also for $n$.) For \eqref{system}, however, it turns out that the dependence between the parameter $α$ and the bacterial mass is bijective and monotone. 
This is both an important difference to signal-production Keller--Segel systems, cf.  \cite{schaaf} and the boundary concentration results,  especially their method of proof, in \cite{delPino_pistoia_vaira}, and not immediately trivial. Indeed, the largest part of the section dealing with the scalar equation (Section \ref{sec:scalarproblem}) will be Section \ref{sec:massmonotone} which will be concerned with the relation between $α$ and the mass $\io n$. Its core idea will be to examine the derivative of solutions $c$ (or rather of $\io αe^c$) with respect to $α$; but some care is necessary to make this idea rigorous. In Section \ref{sec:system} we will use this dependence along with existence results from Section \ref{sec:scalarproblem-ex} to study the full system \eqref{system}. 

In Section \ref{sec.convexity}, we prove Theorem \ref{thm:main2}. The proof rests on symmetry of the solution and classical characterizations of convexity. 
After that, we will further illustrate \eqref{system} by deriving an implicit representation formula for the solution in the one-dimensional setting (Section \ref{sec:one-D}) and, by numerical results in three dimension showing that $n$ and $c$ are convex for $\Omega=B_1(0)$ as stated by Theorem \ref{thm:main2} (Section \ref{sec:numerics}). Moreover, we compare the stationary solution of \eqref{eq:ctconsumption-parabolic} with a stationary of the chemotaxis-Navier-Stokes equations \eqref{eq:ctns}.


However, we begin by recalling some known, but essential prerequisites: 

\section{Preliminaries}

Maximum principle and Hopf's boundary lemma are tools we will invoke often. We use them in the form of \cite[Theorem 3.5 and Lemma 3.4]{GT} -- but do not cite them here. 

The following regularity result will turn out to be useful: 

\begin{lemma}\label{lem-nadirashvili}
 Let $\Om\sub ℝ^{N}$ be a bounded $C^2$-domain. Then there are $C>0$ and $β_0>0$ such that every function $u\in C^2(\Om)\cap C^1(\Ombar)$ satisfies 
\[
 \norm[C^{β_0}(\Ombar)]{u}\le C \kl{\norm[C^0(\Om)]{u} + \norm[\Lp{∞}]{Δu} + \norm[{\Lp[\partial\Om]{∞}}]{\delny u}}.
\]
\end{lemma}
\begin{proof}
 This is \cite[Thm. 1.1]{Nadirashvili}.
\end{proof}

Naturally, large parts of our analysis will be concerned with elliptic equations with Robin boundary conditions. Their solvability is asserted by 
\begin{lemma}\label{GT-6.31;remarkp.124}
 Let $\Om$ be a $C^{2+β}$ domain in $ℝ^N$ and let $a\in C^{β}(\Ombar)$, $a\le 0$, and $b\in C^{1+β}(∂\Om)$, $b\geq0$ such that
 \begin{align*}
 \text{(i)}\quad b > 0 \qquad & \qquad
 \text{or  at least \qquad (ii)}\quad  a\not\equiv 0 \mbox{ or }b\not\equiv 0.
 \end{align*} 
Then for every $f\in C^{β}(\Ombar)$ and $φ\in C^{1+β}(∂\Om)$, the problem 
\[
(Δ+a)u=f\quad \text{ in } \Om,\qquad\quad 
\delny u+bu=φ\quad \text{ on }∂\Om   
\] has a unique $C^{2+β}(\Ombar)$ solution.
\end{lemma}
\begin{proof}
While (i) is (a special case of) \cite[Theorem 6.31]{GT}, the more general part (ii) corresponds to the remark after said theorem, \cite[p.124]{GT}. 
\end{proof}

Higher-order Schauder estimates for these equations are also available:
\begin{lemma}\label{GT-6.30}
	Let $\Om$ be a $C^{2+β}$ domain in $ℝ^{N}$, and let $u\in C^{2+β}(\Ombar)$ be a solution in $\Om$ of $-\Delta u+au=f$ satisfying the boundary condition 
	\[
	B(x)u \equiv b(x)u + \delny u  = φ(x), \qquad x\in \partial\Om.
	\]
	It is assumed that  $f\in C^{β}(\Ombar)$, $φ\in C^{1,β}(\Ombar)$, $a\in C^{β}(\Ombar)$ and $b\in C^{1,β}(\Ombar)$ with 
	\[
	\|a\|_{C^{β}(\Ombar)}, \|b\|_{C^{1+β}(\Ombar)}, \|ν\|_{C^{1+β}(\Ombar)}\le \Lambda.
	\]
	Then 
	\[
	\|u\|_{C^{2+β}(\Ombar)}\le C(\|u\|_{C^{0}(\Ombar)}+\|φ\|_{C^{1+β}(\Ombar)}+\|f\|_{C^{β}(\Ombar)}),
	\]
	where $C=C(N,β,\Lambda,\Om)$.
\end{lemma}
\begin{proof}
This is part of \cite[Theorem 6.30]{GT}.
\end{proof}

\section{The scalar equation}\label{sec:scalarproblem}
\subsection{\textit{A priori} estimates and solvability}\label{sec:scalarproblem-ex}

We will first deal with the single scalar equation that will turn out to be equivalent to the system we are interested in (see Lemma \ref{unique-up-to-multiple}). The first objective is its solvability, to be proven based on a Schauder fixed-point argument, which we prepare by providing \textit{a priori} estimates for solutions to  

\begin{equation}\label{bvp-forSchauder}
\begin{cases}
 Δc = αce^{\ctilde}\qquad& \text{ in }\Om,\\
 \delny c= (γ-c)g\qquad& \text{ on }\partial\Om.
\end{cases}
\end{equation}

Of course, facts derived for this system also apply to \eqref{scalarbvp} if we just insert $\ctilde=c$. 

\begin{lemma}\label{notconstant}
 Let $α\ge 0$, $γ\ge 0$ and let $\Om$ be as in \eqref{cond:Om} and $g$ as in \eqref{condition:g1}.
 If a function $c\in C^2(\Om)\cap C^1(\Ombar)$ solves \eqref{bvp-forSchauder} for some $\ctilde\in C^{β}(\Ombar)$, $β\in(0,1)$, then $c$ is not constant, unless $αγ=0$ and $c\equiv γ$.
\end{lemma}
\begin{proof}
 If $c$ were constant, the boundary condition in \eqref{bvp-forSchauder} would imply \(0=(γ-c)g\) on $\partial\Om$ and, due to $g\not\equiv 0$, hence $c\equiv γ$. But $c\equiv γ$ does not solve $Δc=αce^{\ctilde}$, unless $αγ=0$. 
\end{proof}

\begin{lemma}\label{positivity}
 Let $α\ge 0$, $γ\ge 0$ and $\Om$ and $g$ as in \eqref{cond:Om} and \eqref{condition:g1}.
 If $c\in C^2(\Om)\cap C^1(\Ombar)$ solves \eqref{bvp-forSchauder} for some $\ctilde\in C^{β}(\Ombar)$, $β\in(0,1)$, then $c>0$ in $\Ombar$ or $c\equiv 0$ and $αγ=0$. 
\end{lemma}
\begin{proof}
 The function $(-c)$, being a solution to $L(-c)=0$ for $L=Δ-αe^{\ctilde}$, is either constant (which by Lemma \ref{notconstant} results in $c\equiv γ>0$ or $c\equiv γ$ and $αγ=0$) or cannot achieve a non-negative maximum in $\Om$ \cite[Thm. 3.5]{GT}, which would entail existence of $x_0\in \partial\Om$ such that $(-c)(x_0)>(-c)(x)$ for all $x\in \Om$. If then $-c(x_0)$ were non-negative, by Hopf's boundary lemma \cite[L. 3.4]{GT} we could conclude $\delny(-c)(x_0)>0$. The boundary condition in \eqref{scalarbvp} and non-negativity of $g$ would turn this into 
\[
 0 < (c(x_0)-γ)g(x_0) \le (0-γ)g(x_0)=-γ g(x_0) \le 0,
\]
 which is contradictory and hence proves that $\max_{\Ombar} (-c)$ is negative.
\end{proof}

\begin{lemma}\label{c-between-zero-and-gamma}
 Let $α\ge 0$, $γ\ge 0$ and  let $\Om$ be as in \eqref{cond:Om} and $g$ as in \eqref{condition:g1}. If $c\in C^2(\Om)\cap C^1(\Ombar)$ solves \eqref{bvp-forSchauder} for some $\ctilde\in C^{β}(\Ombar)$, $β\in(0,1)$, then $0<c(x)<γ$ for all $x\in \Ombar$ or $c\equiv γ$ and $αγ=0$.
\end{lemma}
\begin{proof}
 Applying the strong maximum principle \cite[Thm. 3.5]{GT} to the uniformly elliptic operator $L:=Δ-αe^{\ctilde}$, we can conclude that either $c$ is constant (and hence, by Lemma \ref{notconstant}, $c\equiv γ$ and $αγ=0$), or $c$ cannot achieve its (according to Lemma \ref{positivity}, necessarily nonnegative) maximum in the interior of $\Om$ so that there must be $x_0\in \partial\Om$ satisfying $c(x_0)>c(x)\ge 0$ for all $x\in \Om$. By Hopf's boundary point lemma \cite[L. 3.4]{GT}, hence $\delny c(x_0)>0$, which in light of the boundary condition in \eqref{scalarbvp} entails that $ (\gamma-c(x_0)) g(x_0) >0$ and therefore $c(x_0)<γ$. The lower bound has been proven in Lemma \ref{positivity}. 
\end{proof}

\begin{lemma}\label{hoelderestimate}
 Let $\Om$ satisfy \eqref{cond:Om}. For every $β\in(0,\min\set{β_0,β_*})$, with $β_0$ from Lemma \ref{lem-nadirashvili}, 
there is $K>0$ such that for any $α\ge 0$ and $γ\ge 0$ and any $g$ as in \eqref{condition:g1}, every solution $c\in C^2(\Om)\cap C^1(\Ombar)$ of \eqref{scalarbvp} satisfies 
\[
 \norm[C^{β}(\Ombar)]{c} \le K \kl{ γ+αγe^{γ}+γ\norm[{\Lp[\partial\Om]{∞}}]{g} }.
\]
\end{lemma}
\begin{proof}
 According to Lemma \ref{c-between-zero-and-gamma}, applied with $\ctilde=c$, we have $0\le c\le γ$, hence $Δc(x)=αc(x)e^{c(x)}\in [0,αγe^{γ}]$ for all $x\in \Om$. Moreover, $\delny c(x)=g(x) (γ-c(x))\in [0, γg(x)]$ for all $x\in \partial\Om$. An application of Lemma \ref{lem-nadirashvili} immediately concludes the proof. 
\end{proof}

\begin{lemma}\label{solvability-schauderbvp-and-c2betaestimate}
 Let $\Om$ satisfy \eqref{cond:Om} and $β\in(0,β_*)$. For every $α\ge 0$, $γ\ge 0$, $g$ as in \eqref{condition:g1} and $\ctilde\in C^{β}(\Ombar)$, the boundary value problem \eqref{bvp-forSchauder} has a unique solution $c\in C^{2+β}(\Ombar)$.\\
Moreover, for every $K>0$ there is $C>0$ such that for every $α\in[0,K]$, $γ\in[0,K]$, $g$ as in \eqref{condition:g1} with $\norm[C^{1+β}(\Ombar)]{g}\le K$ and $\ctilde\in C^{β}(\Ombar)$ with $\norm[C^{β}(\Ombar)]{\ctilde}\le K$ we have that the solution $c$ of \eqref{bvp-forSchauder} satisfies 
\[
 \norm[C^{2+β}(\Ombar)]{c} \le C.
\]
\end{lemma}
\begin{proof}
 Existence and uniqueness of a solution are asserted by Lemma \ref{GT-6.31;remarkp.124}. 
The Schauder estimate of Lemma \ref{GT-6.30} 
 enables us to find a constant $C_0=C_0(K)>0$ such that any solution $c$ of \eqref{bvp-forSchauder} satisfies $\norm[C^{2+β}(\Ombar)]{c} \le C_0 (\norm[C^0(\Ombar)]{c}+\norm[C^{1+β}(\Ombar)]{γg})$. An application of Lemma \ref{c-between-zero-and-gamma} and the assumptions on $γ$ and $g$ turn this into the desired estimate. 
\end{proof}

\begin{lemma}\label{scalareq-solvable}
 Let $\Om$ be as in \eqref{cond:Om}, $α\ge 0$, $γ\ge 0$, $g$ as in \eqref{condition:g1} and $β\in(0,\min\set{β_*,β_0}]$ with $β_0$ as in Lemma \ref{lem-nadirashvili}. 
 Then \eqref{scalarbvp} has a solution $c\in C^{2+β}(\Ombar)$. 

Moreover, for every $K>0$ there is $C>0$ such that for every $α\in[0,K]$, $γ\in[0,K]$, $g$ as in \eqref{condition:g1} with $\norm[C^{1+β}(\Ombar)]{g}\le K$ we have that every solution $c$ of \eqref{scalarbvp} satisfies 
\begin{align} \label{c2betaestimate}
 \norm[C^{2+β}(\Ombar)]{c} \le C.
\end{align}
\end{lemma}
\begin{proof}
 We let $Φ\colon C^{β}(\Ombar)\to C^{β}(\Ombar)$ be the function that maps $\ctilde\in C^{β}(\Ombar)$ to the solution $c$ of \eqref{bvp-forSchauder}. By Lemma \ref{solvability-schauderbvp-and-c2betaestimate}, this function is well-defined and, moreover, 
 compact. In order to prepare an application of the Leray--Schauder fixed point theorem, we assume that $σ\in[0,1]$ and $c\in C^{β}(\Ombar)$ are such that $c=σΦ(c)$. Then $Δc=σΔΦ(c)=ασΦ(c)e^c=αce^c$ in $\Om$ and $\delny c = σ\delny Φ(c) = σ(γ-Φ(c))g =  (σγ-c)g$ on $\partial\Om$. \\
 According to Lemma \ref{c-between-zero-and-gamma}, $c$ thus satisfies $0\le c\le σγ$ in $\Ombar$. 
 With $C$ from Lemma \ref{lem-nadirashvili}, we hence obtain that 
\begin{align}
 \norm[C^{β}(\Ombar)]{c} &\le C\kl{ \norm[C^0(\Ombar)]{c} + \norm[\Lp{∞}]{αce^c} +\norm[{\Lp[\partial\Om]{∞}}]{ (σγ-c)g} } \nn \\ 
 &\le C \kl{σγ + ασγe^{σγ} + σγ \norm[{\Lp[\partial\Om]{∞}}]{g}} \le Cγ(1+αe^{γ}+ \norm[{\Lp[\partial\Om]{∞}}]{ g } ).\label{fixedpt-hoelder}
\end{align}
Due to the Leray--Schauder theorem \cite[Thm. 10.3]{GT}, there is a fixed point $c\in C^{β}(\Ombar)$ of $Φ$. 
The $C^{2+β}(\Ombar)$ estimate \eqref{c2betaestimate} results from \eqref{fixedpt-hoelder} and the second part of Lemma \ref{solvability-schauderbvp-and-c2betaestimate}, applied with $\ctilde=c$. 
\end{proof}

\subsection{Dependence on $α$, part I: monotonicity (and uniqueness)}

Having ensured that \eqref{scalarbvp} is solvable for any parameter $α$, we can now turn our attention to the dependence of the solution on this parameter. Apparently, this will provide crucial information for the investigation of uniqueness of the system. We begin by revealing monotonicity of $c$ with respect to $α$: 

\begin{lemma}\label{monotonewrtalpha}
 Let $\Om$ satisfy \eqref{cond:Om}, $γ\ge 0$, let $g$ be as in \eqref{condition:g1}. Assume that $α_1\ge α_2>0$ or $α_1>α_2\ge 0$. If we let $c_{α_1}, c_{α_2} \in C^2(\Om)\cap C^1(\Ombar)$ denote solutions to \eqref{scalarbvp} with $α=α_1$ and $α=α_2$, respectively, then 
\[
 c_{α_1} \le c_{α_2} \qquad \text{ in } \Ombar.
\]
\end{lemma}
\begin{proof}
 Letting $\ctilde:=c_{α_1}-c_{α_2}$, we define $\Omtilde:=\set{x\in \Om \mid \ctilde(x)>0}$ and note that $\Omtilde$, which can be assumed to be connected without loss of generality, is open. We assume that $\Omtilde$ is non-empty. Letting $Γ_1:=(\partial\Omtilde\cap \partial\Om)^\circ$, with the interior taken with respect to the relative topology of $\partial\Om$, and $Γ_2:=\partial\Omtilde\setminus Γ_1 = \overline{\partial\Omtilde\setminus\partial\Om}$, we see that $\ctilde$ satisfies $\ctilde|_{Γ_2}=0$ and $\delny \ctilde = - g \ctilde$ on $Γ_1$. The normal on $Γ_1$ coincides with the normal of $\partial\Om$. 
 From $α_1\ge α_2$, the monotonicity of $ξ\mapsto ξe^{ξ}$ on $[0,∞)$ and nonnegativity of $c_{α_1}, c_{α_2}$ according to Lemma \ref{c-between-zero-and-gamma}, we conclude that 
\[
 Δ\ctilde = α_1c_{α_1}e^{c_{α_1}} - α_2c_{α_2}e^{c_{α_2}} \ge α_2 ( c_{α_1}e^{c_{α_1}} - c_{α_2}e^{c_{α_2}}) >0 \text{ in }\Omtilde
\]
if $α_2>0$, or, if $α_1>α_2=0$, 
\[
 Δ\ctilde = α_1c_{α_1}e^{c_{α_1}} - α_2c_{α_2}e^{c_{α_2}} > α_2 ( c_{α_1}e^{c_{α_1}} - c_{α_2}e^{c_{α_2}}) = 0 \text{ in }\Omtilde. 
\]
Since strict positivity of $Δ\ctilde$ shows that $\ctilde$ is not constant, the maximum principle \cite[Thm. 3.5]{GT} entails that there is $x_0\in \partial\Omtilde$ such that $\ctilde(x_0)>\ctilde(x)$ for all $x\in \Omtilde$. Necessarily, $x_0\in Γ_1$, because $\ctilde|_{Γ_2}=0$. From Hopf's lemma \cite[L. 3.4]{GT}, 
\[
 0 < \delny \ctilde(x_0) = -g(x_0) \ctilde(x_0),
\]
so that $\ctilde(x_0)<0$, in contradiction to the definition of $\Omtilde$ and continuity of $\ctilde$. Hence $\Omtilde=\emptyset$ and, accordingly, $c_{α_1}\le c_{α_2}$ throughout $\Ombar$. 
\end{proof}

A first, important consequence of this monotonicity is uniqueness of solutions: 

\begin{lemma}\label{uniqueness}
 Let $\Om$ be as in \eqref{cond:Om}, $α> 0$, $γ\ge 0$, $g$ as in \eqref{condition:g1}. Then the solution to \eqref{scalarbvp} is unique in $C^2(\Om)\cap C^1(\Ombar)$. 
\end{lemma}
\begin{proof} 
We can apply Lemma \ref{monotonewrtalpha} with $α_1=α=α_2$. 
\end{proof}

\subsection{Dependence on $α$, part II: Monotonicity of the mass}\label{sec:massmonotone} 

If we want to conclude uniqueness of solutions to \eqref{system} from unique solvability of \eqref{scalarbvp}, we will be required to have determined $α$ uniquely. (This is a step that does not hold true in the classical Keller--Segel system.) To reach this objective, we will rely on the relation $\io n = \io αe^c$ between the bacterial mass and $α$. In order to prepare the necessary differentiation of $c$, let us  introduce the following auxiliary functions: 

Given $γ> 0$, $\Om$ as in \eqref{cond:Om} and $g$ as in \eqref{condition:g1}, 
for $α_1,α_2\in[0,∞)$ with $α_2\neq α_1$ we define 
\begin{equation}\label{def:w}
 w_{α_2,α_1} := \f{c_{α_2}-c_{α_1}}{α_2-α_1}, 
\end{equation}
where by $c_{α_1}$ and $c_{α_2}$ we denote the solution to \eqref{scalarbvp} with $α=α_1$ or $α=α_2$, respectively.

Moreover, we define 
\begin{equation}\label{def:f1}
 f_{1,α_2} := c_{α_2}e^{c_{α_2}}
\end{equation}
and 
\begin{equation}\label{def:f2}
 f_{2,α_1,α_2} := α_1e^{c_{α_1}}+α_1c_{α_1}e^{c_{α_1}}F(c_{α_2}-c_{α_1}),
\end{equation}
where $F$ is the nonnegative, analytic function defined by 
\begin{equation}\label{def:F}
F(z)=\begin{cases}\f{e^z-1}{z}& \text{for } z\neq 0,\\
 F(0)=1. 
     \end{cases}
\end{equation}
The reason for the above choice of $f_{1,α_2}$ and $f_{2,α_1,α_2}$ should become clear in the following lemma: 

\begin{lemma}
 Let $\Om$ be as in \eqref{cond:Om}, $γ>0$, $g$ as in \eqref{condition:g1} and $α_1,α_2\in [0,\infty)$ with $α_1\neq α_2$. 
 Then the function $w_{α_2,α_1}$ from \eqref{def:w} satisfies 
\begin{equation}\label{equationforwalpha}
\begin{cases}
 Δw_{α_2,α_1} = f_{1,α_1} + f_{2,α_1,α_2} w_{α_2,α_1}\;\;&\text{in } \Om, \qquad\\
 \delny w_{α_2,α_1} = -g w_{α_2,α_1} &\text{ on } \dOm,
\end{cases}
\end{equation}
with $f_{1,α_2}$ as in \eqref{def:f1} and $f_{2,α_1,α_2}$ from \eqref{def:f2}.
\end{lemma}

\begin{proof}
If we use \eqref{scalarbvp}, we see that in $\Om$ we have 
\begin{align*}
 Δ&(c_{α_2}-c_{α_1})\\
 &= α_2c_{α_2}e^{c_{α_2}}-α_1c_{α_1}e^{c_{α_1}} \\
 &= (α_2-α_1) c_{α_2}e^{c_{α_2}} + α_1(c_{α_2}-c_{α_1})e^{c_{α_2}} + α_1c_{α_1}(e^{c_{α_2}}-e^{c_{α_1}})\\
 &= (α_2-α_1) c_{α_2}e^{c_{α_2}} + α_1(c_{α_2}-c_{α_1})e^{c_{α_2}} + α_1c_{α_1}e^{c_{α_1}} (e^{c_{α_2}-c_{α_1}}-1)\\
 &= (α_2-α_1) c_{α_2}e^{c_{α_2}} + α_1e^{c_{α_2}}(c_{α_2}-c_{α_1}) + α_1c_{α_1}e^{c_{α_1}} F(c_{α_2}-c_{α_1}) (c_{α_2}-c_{α_1}),
\end{align*}
 and division by $α_2-α_1$ together with \eqref{def:f1} and \eqref{def:f2} shows \eqref{equationforwalpha}. Also the boundary condition results from \eqref{scalarbvp} in a straightforward manner. 
\end{proof}

We already know the sign of solutions to \eqref{equationforwalpha}: 
\begin{lemma}\label{w-negative}
 Let $\Om$ be as in \eqref{cond:Om}, $γ>0$, $g$ as in \eqref{condition:g1} and $α_1,α_2\in [0,\infty)$ with $α_1\neq α_2$. 
 Then $w_{α_2,α_1}\le 0$.
\end{lemma}
\begin{proof}
 According to \eqref{def:w}, this is an immediate corollary of Lemma \ref{monotonewrtalpha}.
\end{proof}

An estimate in the other direction is what we obtain next: 

\begin{lemma}\label{wtildeboundedbelow}
 Let $\Om$ satisfy \eqref{cond:Om} and $g$ \eqref{condition:g1}. Let $γ>0$. 
 For any $β\in(0,β_*)$, 
the boundary value problem 
\begin{equation}\label{wtilde}
 Δ\wtilde = γe^{γ}, \qquad \delny \wtilde = -g \wtilde
\end{equation}
 has a unique solution $\wtilde\in C^{2+β}(\Ombar)$, which satisfies that $\wtilde\le w_{α_2,α_1}\le 0$ for any choice of $α_1,α_2\in[0,∞)$ with $α_1\neq α_2$.
\end{lemma}
\begin{proof}
 Unique solvability is ensured by Lemma \ref{GT-6.31;remarkp.124}. 
 From Lemma \ref{c-between-zero-and-gamma}, we know that $0\le c_{α_2} \le γ$ and hence $0\le f_{1,α_2}\le γe^{γ}$. 
 Furthermore, $f_{2,α_1,α_2}\ge 0$, so that from non-positivity of $w_{α_2,α_1}$ according to Lemma \ref{w-negative}, we can conclude non-negativity of 
\[
 Δ ( \wtilde-w) = γe^{γ} - f_{1,α_2} - f_{2,α_1,α_2} w. 
\]
 Letting $x_0\in \Ombar$ be such that $(\wtilde-w)$ obtains a maximum at $x_0$, we will derive a contradiction from $(\wtilde - w)(x_0)\ge 0$. We can assume that either $x_0\in \partial \Om$ and $(\wtilde-w)(x_0)>(\wtilde-w)(x)$ for every $x\in \Om$, which according to \cite[L. 3.4]{GT} entails positivity of $\delny(\wtilde-w)(x_0) = -g(x_0) (\wtilde- w)(x_0)$ and hence negativity of $(\wtilde - w)(x_0)$, or that $x_0\in \Om$, which, according to the maximum principle \cite[Thm. 3.5]{GT} is only possible if $\wtilde-w$ is constant. But then $γe^{γ}-f_{1,α_2}$ and $f_{2,α_1,α_2} w$ both have to be zero, resulting in $c_{α_2}\equiv γ$ and either $c_{α_1}=c_{α_2}\equiv γ$ or $α_1=0$. Since $γ>0$, this is only possible if $α_1=0=α_2$ (cf. Lemma \ref{c-between-zero-and-gamma}), contradicting the assumption $α_1\neq α_2$. Hence, $\wtilde<w$. 
\end{proof}

An important purpose of these pointwise estimates for $w$ lies in serving as groundwork for estimates in better spaces, thus preparing the application of Arzel\`a--Ascoli type arguments.

\begin{lemma}\label{c2hoelderforw}
 Let $\Om$ and $g$ be as in \eqref{cond:Om} and \eqref{condition:g1} and $γ>0$. Then for every $β\in(0,min\set{β_*,β_0})$, with $β_0$ from Lemma \ref{lem-nadirashvili}, the following holds: 
 For every $K>0$ there is a constant $C>0$ such that for any choice of $α_1,α_2\in [0,K]$ with $α_1\neq α_2$, the function $w_{α_2,α_1}$ defined in \eqref{def:w} satisfies 
\begin{equation}\label{eq:c2hoelderforw}
 \norm[C^{2+β}(\Ombar)]{w_{α_2,α_1}}\le C. 
\end{equation}
\end{lemma}
\begin{proof}[Proof of Lemma \ref{c2hoelderforw}]
 According to Lemma \ref{GT-6.30}
, for every $\Lambda>0$, there is $M>0$ such that whenever 
 \[
  \norm[C^{β}(\Ombar)]{φ} \le \Lambda, \qquad \norm[C^{1+β}(\Ombar)]{g}\le \Lambda, \qquad \norm[C^{1+β}(\Ombar)]{ν} \le \Lambda 
 \]
 any solution $w$ of $(Δ-φ)w=f$ in $\Om$, $ gw+\delny w=0$ on $\dOm$ satisfies 
\[
 \norm[C^{2+β}(\Ombar)]{w}\le M (\norm[L^\infty(\Om)]{w}+\norm[C^{β}(\Ombar)]{f}).
\]
 With 
\[
 \Lambda:=\max\set{\sup_{α_1,α_2\in [0,K]}\norm[C^{β}(\Ombar)]{f_{2,α_1,α_2}}, \norm[C^{1+β}(\Ombar)]{g },\norm[C^{1+β}(\Ombar)]{ν}},
\]
 which is finite due to \eqref{cond:Om}, \eqref{condition:g1} and a combination of \eqref{def:f2} with Lemma \ref{hoelderestimate}, we can apply this estimate, deriving that for every $α_1,α_2\in [0,K]$ with $α_1\neq α_2$, 
\[
 \norm[C^{2+β}(\Ombar)]{w_{α_2,α_1}}\le M (\norm[L^\infty(\Om)]{w_{α_2,α_1}}+\norm[C^{β}(\Ombar)]{f_{1,α_2}}).
\]
 Using that, again by Lemma \ref{hoelderestimate}, also $\sup_{α_2\in [0,K]} \norm[C^{β}(\Ombar)]{f_{1,α_2}}$ is finite, as is $\norm[\Lp\infty]{w_{α_2,α_1}}$ due to Lemma \ref{wtildeboundedbelow}, we obtain \eqref{eq:c2hoelderforw} with $C = M(\norm[L^\infty(\Om)]{\wtilde}+\sup_{α_2\in [0,K]} \norm[C^{β}(\Ombar)]{f_{1,α_2}})$.
\end{proof}

For obtaining the convergence of $w_{α_2,α_1}$ as $α_2\to α_1$, the mere extraction of a convergent subsequence, which has been prepared by Lemma \ref{c2hoelderforw}, is insufficient. Fortunately for the identification of its limit, Lemma \ref{c2hoelderforw} has another immediate consequence  pertaining to the continuity of the terms in \eqref{equationforwalpha} with respect to $α$: 

\begin{cor}\label{alphamapstocalphacontinuos}
	Let $\Om$ satisfy \eqref{cond:Om}, $g$ be as in \eqref{condition:g1} and $γ>0$. 
	The map 
	\[
	Γ\colon
	\begin{cases}
	 [0,\infty)\to C^2(\Ombar)\\
	α\mapsto c_{α}
	\end{cases}
	\]
	is continuous. 
\end{cor}
\begin{proof}
 If we insert the definition of $w_{α_2,α_1}$, \eqref{eq:c2hoelderforw} turns into 
\begin{equation*}
 \norm[C^{2+β}(\Ombar)]{c_{α_2}-c_{α_1}}\le C|α_2-α_1|, 
\end{equation*}
ensuring Lipschitz continuity of $Γ$.
\end{proof}




Now it is time to show that $c_{α}$ is differentiable with respect to $α$ and to characterize the derivative: 

\begin{lemma}\label{differentiable}
 Let $\Om$ satisfy \eqref{cond:Om}, $g$ be as in \eqref{condition:g1} and $γ>0$. 
 For every $α>0$, 
 the function 
\begin{equation}\label{def:calphaprime}
 c_{α}' := \lim_{α_2\to α} w_{α_2,α}
\end{equation}
 exists as limit in $C^2(\Ombar)$ 
and is the unique solution of  
\begin{equation}\label{eq:calphaprime}
 \begin{cases} Δc_{α}' = c_{α} e^{c_α} + (αe^{c_{α}} + αc_{α}e^{c_{α}}) c_{α}'\qquad &\text{ in } \Om,\\
   \delny c_{α}' = -g c_{α}'\qquad &\text{ on } \dOm.
 \end{cases}
\end{equation}
\end{lemma}
\begin{proof}
 We let $α>0$ and $K:=α+1$, so that from Lemma \ref{c2hoelderforw}, we obtain $C>0$ such that for all $α_2\in [0,K]$ 
\begin{equation}\label{boundsforconvergence}
 \norm[C^{2+β}(\Ombar)]{w_{α_2,α}} \le C.
\end{equation}
 If then $w_{α_2,α}$ were not convergent to the solution $c_{α}'$ of \eqref{eq:calphaprime} as $α_2\to α$, we could find $ε_0>0$ and a sequence $\folge[k]{α_{2,k}}\subset [0,K]$ with limit $α$ such that $\norm[C^2(\Ombar)]{w_{α_{2,k},α}-c_{α}'}>ε_0$ for every $k\in ℕ$. However, according to \eqref{boundsforconvergence} and Arzel\`a--Ascoli's theorem, for a suitably chosen subsequence $\folge[j]{α_{2,k_j}}$, $w_{α_{2,k_j},α}$ converges in $C^2(\Ombar)$ with a limit $w$. 
 By Corollary \ref{alphamapstocalphacontinuos}, we have that $\lim_{α_2\to α} f_{2,α,α_2} = αe^{c_{α}} + αc_{α}e^{c_{α}}$ exists (as limit in $C^2(\Ombar)$), so that $w$ would have to solve 
\begin{equation*}
 \begin{cases} Δw = c_{α} e^{c_α} + (αe^{c_{α}} + αc_{α}e^{c_{α}}) w\\
   \delny w = -g w.
 \end{cases}
\end{equation*}
 But according to Lemma \ref{GT-6.31;remarkp.124}, 
 this problem has a unique solution, i.e. $c_{α}'$, which would contradict the choice of $\folge[k]{α_{2,k}}$. 
\end{proof}

The following estimate gives exactly the quantitative control on $c_{α}'$ that we will need: 

\begin{lemma}\label{calphaprime-estimate}
 Assuming that $\Om$ satisfies \eqref{cond:Om}, $g$ is as in \eqref{condition:g1} and $γ>0$, we let $α>0$.  
 Then the function $c_{α}'$ satisfies 
 \[
  0\ge c_{α}' > -\f1{α}. 
 \]
\end{lemma}

\begin{proof}
 We let $ζ_{k}$ be a solution (whose existence is guaranteed by Lemma \ref{GT-6.31;remarkp.124} (i)) to 
\[
 \bigg(Δ-(αe^{c_{α}} + αc_{α}e^{c_{α}})\bigg) ζ_{k}  = c_{α} e^{c_α}\;\text{in } \Om,\qquad  \delny   ζ_{k} =  -\left(g+\f1{k}\right) ζ_{k}\; \text{on } \dOm.
 \]
 In light of Lemma \ref{GT-6.30} 
 and Arzel\`a--Ascoli's theorem together with unique solvability of \eqref{eq:calphaprime}, it is easy to see that $ζ_{k}\to c_{α}'$ in $C^2(\Ombar)$ as $k\to \infty$. 
For any $k\in ℕ$, there is either $x_0\in \dOm$ such that $ζ_{k}(x_0)<ζ_{k}(x)$ for all $x\in \Om$ or there is $x_0\in \Om$ such that $ζ_{k}$ attains its minimum at $x_0$. In the former case, we are dealing with a minimum on the boundary and hence $0 \ge \delny ζ_{k}(x_0) = -(g(x_0)+\f1{k}) ζ_{k}(x_0)$, which due to negativity of $-(g(x_0)+\f1{k})$ shows that $ζ_{k}(x_0)\ge 0$, so that $ζ_{k}\ge 0$ in $\Ombar$. 
In the latter case, $Δζ_{k}(x_0)\ge 0$, i.e. 
\[
 c_{α}(x_0) e^{c_α(x_0)} + (αe^{c_{α}(x_0)} + αc_{α}(x_0)e^{c_{α}(x_0)}) ζ_{k}(x_0) \ge 0
\]
and hence 
\[
  c_{α}(x_0) + (α + αc_{α}(x_0)) ζ_{k}(x_0) \ge 0, 
\]
meaning that 
\[
 ζ_{k} (x_0) \ge - \f{c_{α}(x_0)}{α (1 + c_{α}(x_0))} \ge - \f1{α(1+\f{1}{\sup c_{α}})}. 
\]
Passing to the limit $k\to \infty$, we infer 
\[
 0\ge c_{α}' \ge - \f1{α(1+\f{1}{\sup c_{α}})} > -\f1{α}\qquad \text{in } \Om, 
\]
where the first inequality is due to the defintion of $c_\alpha'=\lim_{α_2\to α}w_{\alpha_2,\alpha}$ and nonpositivity of $w_{α_2,α}$ by Lemma \ref{w-negative}.
\end{proof}

These preparations about $c_{α}'$ culminate in the following statement concerning the dependence of the bacterial mass on $α$: 

\begin{lemma}\label{alpha-to-mass-injective}
 Let $\Om$ satisfy \eqref{cond:Om}, $g$ be as in \eqref{condition:g1} and $γ>0$. 
 
 The map 
\[
m\colon \begin{cases} 
[0,\infty)\to [0,\infty)\\
  α\mapsto α\int e^{c_α}
 \end{cases}
\]
 is continuous in $[0,\infty)$, differentiable in $(0,\infty)$, monotone increasing, and surjective, hence bijective.  
\end{lemma}
\begin{proof}
 Letting $α>0$ and $α_2\in(0,∞)\setminus\set{α}$, we see that with $F$ from \eqref{def:F}
\begin{align*}
 \f{m(α_2)-m(α)}{α_2-α} &= \io \f{α_2e^{c_{α_2}}-αe^{c_{α}}}{α_2-α} \\
 &= \io e^{c_{α}} + \io αe^{c_{α}} \f{e^{c_{α_2}-c_{α}}-1}{α_2-α}\\
 &= \io e^{c_{α}} + \io αe^{c_{α}} F(c_{α_2}-c_{α}) w_{α_2,α},
\end{align*}
  where we can pass to the limit $α_2\to α$ easily, thanks to Corollary \ref{alphamapstocalphacontinuos} and Lemma \ref{differentiable}, obtaining the existence of
 \[
  m'(α)= \io e^{c_{α}} + \io αe^{c_{α}} c_{α}' = \io e^{c_{α}}(1 + α c_{α}').
 \]
 The lower estimate $c_{α}'>-\f1{α}$ from Lemma \ref{calphaprime-estimate} shows that $m'$ is positive. Surjectivity results from the trivial estimate 
\[
 m(α) = α\io e^{c_α} \ge α\io e^0 = α|\Om|. \qedhere
\]
\end{proof}

\section{The system: Existence and uniqueness} \label{sec:system}

We now want to employ the information on the scalar equation \eqref{scalarbvp} obtained in the previous section for solving the actual system \eqref{system}. In order to make sure that each can be transformed into the other, we look at the first equation of \eqref{system}: 

\begin{lemma}\label{unique-up-to-multiple}
	Let $\Om$ be a bounded domain and $c\in C^2(\Om)\cap C^1(\Ombar)$. Assume that $n\in C^2(\Om)\cap C^1(\Ombar)$ satisfies
		\begin{equation}\label{n-eq}
	\begin{cases}
	0 = Δ n -\nabla\cdot (n\nabla c)\qquad\text{in } \Om\\
	\delny n = n\delny c\qquad\text{on } \dOm.
	\end{cases}
	\end{equation}
	 Then there is $\alpha\in\mathbb R$ such that 
\begin{equation} \label{relationbetweennandc}
n=\alpha e^c.
\end{equation}
\end{lemma}
\begin{proof}
	For any $c$ with the assumed regularity, $e^c$ is a positive element of $C^2(\Om)\cap C^1(\Ombar)$ and 
	\[
	Δ e^c - \nabla \cdot (e^c\nabla c) = \nabla \cdot (e^c\nabla c)-\nabla \cdot (e^c \nabla c) = 0\qquad \text{in } \Om. 
	\]
 If $n\equiv0$, the assertion is trivial with $\alpha=0$. Note that $-n$ is also a solution of \eqref{n-eq}, therefore we can assume without loss of generality that there exists a point $x_0\in\Omega$ such that $n(x_0)>0$. Thus, we have
  \begin{align}\nonumber
 4\bigg|\nabla\sqrt{\frac{n}{ e^c}}\bigg|^2 e^c 
 &=
 \bigg|\nabla \frac{n}{e^c}\bigg|^2\frac{e^{2c}}{n}
 =
 \left|\frac{\nabla n}{e^c}-\frac{n\nabla e^c}{e^{2c}}\right|^2\frac{e^{2c}}{n}
 \\&=\nonumber
 \left(\frac{|\nabla n|^2}{n}-\frac{\nabla n\cdot\nabla e^c}{e^c}\right)+\left(\frac{n|\nabla e^c|^2}{e^{2c}}-\frac{\nabla n\cdot\nabla e^c}{e^c}\right)
 \\&=\nonumber
 \nabla n\cdot \nabla\log\frac{n}{e^c}-\nabla e^c \cdot \nabla\frac{n}{e^c}
 \\&=
 \nabla n\cdot \nabla\log\frac{n}{\epsilon e^c}-\epsilon\nabla e^c \cdot \nabla\frac{n}{\epsilon e^c} \label{computation.42}
 \end{align}
 in $\Om_+:=\set{x\in \Om\mid n(x)>0}$ for every $\epsilon>0$. 
The most useful form in which to use the equation for $n$ will be the weak version of \eqref{n-eq}:  
Each function $\rho\in \set{n, e^c}$ satisfies 
 \begin{equation}\label{weakeq}
 \io \na \rho \cdot \na \phi = \io \rho\na c\cdot \na \phi \qquad \text{for every } \phi\in H^1(\Om). 
 \end{equation}
 For $\epsilon>0$ we consider $\psi_\epsilon:=\max\{\frac{n}{\epsilon e^c},1\}$. We have that $\psi_\epsilon(x_0)\neq1$ if $\epsilon<\frac{n(x_0)}{e^{c(x_0)}}$. Note that $\psi_\epsilon$, $\sqrt{\psi_\epsilon}$ and $\log(\psi_\epsilon)$ belong to $H^1(\Omega)$. We therefore are allowed to use $\psi_\epsilon$ and $\log(\psi_\epsilon)$ as test functions in \eqref{weakeq}. By \eqref{computation.42}, it holds that 
 \begin{align*}\nonumber
 4\int_{\{n\geq \epsilon e^c\}}\bigg|\nabla\sqrt{\frac{n}{ e^c}}\bigg|^2 e^c dx
 &=
 \int_{\{n\geq \epsilon e^c\}}\nabla n\cdot \nabla\log\frac{n}{\epsilon e^c}dx-\epsilon\int_{\{n\geq \epsilon e^c\}}\nabla e^c \cdot \nabla\frac{n}{\epsilon e^c}dx
 \\
 &=
 \int_{\Omega}\nabla n\cdot \nabla\log(\psi_\epsilon) dx-\epsilon\int_{\Omega}\nabla e^c \cdot \nabla\psi_\epsilon dx
 \\&
 \stackrel{\eqref{weakeq}}{=}
 \int_{\Omega}n\nabla c\cdot \nabla\log(\psi_\epsilon)dx-\int_{\Omega}e^c\nabla c \cdot \nabla\psi_\epsilon dx
 \\&=
 \int_{\{n\geq \epsilon e^c\}}\nabla c \cdot \left(n\frac{\nabla n}n-n\frac{\nabla e^c}{e^c}-\frac{e^c}{e^c}\nabla n +e^cn\frac{\nabla e^c}{e^{2c}}\right)dx=0.
 \end{align*}	
 By letting $\epsilon\to0$, we obtain that $\frac{n}{e^c}$ is constant on the connected component of $\Om_+$ containing $x_0$. Let $A$ be the connected component of $\Om_+$ that contains $x_0$. We can conclude that then there exists $\alpha>0$ such that $n=\alpha e^c$ on $A$.
As $c$ and $n$ are continuous, we have that $n=\alpha e^c>0$ also holds on $\overline A$. However, this directly implies that $A=\Omega$, which yields the assertion.
\end{proof}

With this, we can prove the main result: 

\begin{proof}[Proof of Theorem \ref{thm:main}]
 Due to Lemma \ref{alpha-to-mass-injective}, there is exactly one number $α\in[0,\infty)$ such that $m(α)=m$. With this $α$, Lemma \ref{scalareq-solvable} ensures solvability of \eqref{scalarbvp}, and setting $n:=αe^c$, we obtain a solution to \eqref{system} with $\io n=m(α)$. \\
On the other hand, if $(n,c)\in (C^{2}(\Om)\cap C^1(\Ombar))^2$ solves \eqref{system} and satisfies $\int_{\Omega}n\ge 0$, then by Lemma \ref{unique-up-to-multiple}, there exists a number $\alpha\geq0$ such that $n = α e^c$. Thus, $c$ solves \eqref{scalarbvp} fulfilling $m=\io n = α\io e^c = m(α)$. Uniqueness of $α$ satisfying $m(α)=m$ and of the solution of \eqref{scalarbvp} with this value of $α$ (according to Lemma \ref{uniqueness}) show uniqueness of the solution to \eqref{system}. \\
 That this solution is not constant for $\io n>0$ implying $\alpha>0$ can be seen from Lemma \ref{notconstant}.
\end{proof}


\section{The shape of the solution}\label{sec.shape}
Having shown existence and uniqueness of solutions, we want to use this section to illustrate some of their qualitative properties and to gain insight into their shape. We begin with the radially symmetric setting and the proof of Theorem \ref{thm:main2} in Section \ref{sec.convexity}. Then we will turn our attention to the one-dimensional setting and the derivation of an implicit representation of the solution, Section \ref{sec:one-D}, and finally in Section \ref{sec:numerics} we will present the results of some numerical experiments.

\subsection{The radial setting and convexity}\label{sec.convexity}
Here we treat the special case, where $\Omega:=B_R$ is the open ball of radius $R$ at $0$. Moreover, we assume that $g$ is constant. Let $\partial_r=\frac{x}{|x|}\cdot\nabla$ be the radial derivative. We can rewrite \eqref{system} into
\begin{align}\label{system.ball}
\begin{cases}
0 = Δn- \nabla\cdot(n\nabla c)&\text{in } B_R \\
0= Δc-nc&\text{in } B_R \\
\partial_r c = (γ-c)g&\text{on } \partial B_R\\
\partial_rn = n\partial_r c&\text{on } \partial B_R
\end{cases}
\end{align}
for some $g>0$ and $γ>0$. For fixed mass $m>0$, this system admits a unique classical solution $(n,c)$ with $n\ge 0$ such that $\io n = m$ according to Theorem \ref{thm:main}. This solution has to be radially symmetric and we may write $n(|x|)=n(x)$ as well as $c(|x|)=c(x)$. Moreover, we have seen in Lemma 
\ref{unique-up-to-multiple} that
\begin{align*}
n(r)=n(R)e^{c(r)-c(R)}
\end{align*}
is satisfied for all $0\leq r\leq R$.

	Foundation of the proof of Theorem \ref{thm:main2} will be the well-known fact that for smooth radially symmetric functions convexity is ensured if the second derivative in radial direction has positive sign. We begin this section with an elementary proof of this fact.

	\begin{lemma}\label{lem.convex}
		Let $N\ge1$, $R>0$ and $\Om=B_R\subset ℝ^N$. Let $u:[0,R)\to \mathbb R$ be differentiable in $0$  and (strictly) convex such that $ u'(0)\geq0$. Then $x\mapsto u(|x|)$ is a (strictly) convex function on $\Om$.
	\end{lemma} 
\begin{proof}
	We assume that $u$ is strictly convex. First, we show that $u$ is monotone. Without loss of generality, we assume that $u(0)=0$, because otherwise we can consider the function $u-u(0)$. Let $r>0$. By the convexity of $u$, we have
	\begin{align*}
	0\leq ru'(0)=\lim\limits_{t\to0}\frac{u(tr)}{t}\leq  u(r).
	\end{align*}
	Thus, $u$ is non-negative. Let $0<r_1<r_2$. The strict convexity of $u$ implies
	\begin{align*}
	u(r_1)&=u\left(\frac{r_1}{r_2}r_2+\left(1-\frac{r_1}{r_2}r_2\right)0\right)
	< \frac{r_1}{r_2}u(r_2)+\left(1-\frac{r_1}{r_2}\right)u(0)
	=\frac{r_1}{r_2}u(r_2)\leq u(r_2),
	\end{align*}
	 which shows that $u$ is strictly monotonously increasing.
	 The next step is to prove that $x\mapsto u(|x|)$ is strictly convex. Let $x,y\in B_R$, $x\neq y$ and $t\in(0,1)$. 
	 
	 Case 1: $|x|=|y|$: Due to the strict convexity of the ball $\overline{B_{|x|}}$, we obtain $|tx+(1-t)y|< t|x|+(1-t)|y|$. Then the strict monotonicity of $u$ yields 
	 \begin{align*}
	 u(|tx+(1-t)y|)< u(t|x|+(1-t)|y|)\leq tu(|x|)+(1-t)u(|y|),
	 \end{align*}
	  where we have used that $u$ is convex.
	  
	  Case 2: $|x|\neq |y|$: The triangle inequality ensures that $|tx+(1-t)y|\leq t|x|+(1-t)|y|$. In this case, we employ the strict convexity of $u$ and the fact that $u$ is monotonically increasing to see that 
	  \begin{align*}
	  u(|tx+(1-t)y|)\leq u(t|x|+(1-t)|y|)< tu(|x|)+(1-t)u(|y|).
	  \end{align*}
	   Thus, $x\mapsto u(|x|)$ is strictly convex. For the case of convexity instead of strict convexity, the same proof applies with each '$<$' replaced by '$\le$'.
\end{proof}

In particular, Lemma \ref{lem.convex} shows (strict) convexity of radially symmetric functions $c\in C^1(\Om)$ that are differentiable at $0$ (and thus automatically satisfy $\partial_rc(0)=0$) and whose radial derivative is (strictly) increasing (which entails convexity of the restriction of $c$ to a radial line). Both is the case for solutions $c$ to \eqref{system.ball}:
	
\begin{lemma}\label{lem.monotone}
	Let $R>0$, $\Om=B_R$, and let $g>0$ and $γ>0$ be constant. Then the solution to \eqref{system.ball} satisfies 
		\[
	 \partial_rc(r)>0\text{ for } r\in(0,R)\text{ and } \partial_rc \text{ is strictly increasing on } (0,R).
	\]
\end{lemma}
\begin{proof}
	Due to the radial symmetry of $c$ and according to Lemma \ref{unique-up-to-multiple}, we can rewrite the equation for $c$ as
	\begin{align}\label{equation-radial}
	\frac{1}{r^{N-1}}\partial_r\left(r^{N-1}\partial_r c(r)\right)=\alpha c(r) e^{c(r)} \qquad \text{for } r\in(0,R)
	\end{align}
	with some $\alpha>0$. Multiplying this by $r^{N-1}$ and integration with respect to $r$ entail that
	\begin{align*}
	\partial_rc(r) =\frac{1}{r^{N-1}} \int_0^r \rho^{N-1} \alpha c(\rho)e^{c(\rho)} d\rho >0 \qquad \text{for every }r\in(0,R),
	\end{align*}
	where we have used that $\partial_rc(0)=0$ and $c>0$.  Substituting $t=\frac\rho{r}$ in the integral, we have
	\begin{align*}
	\partial_rc(r)=r\int_0^1 t^{N-1} \alpha c(rt)e^{c(rt)} dt \qquad \text{for }r\in(0,R).
	\end{align*}
	The positivity and strict monotonicity of $c$ imply that $r\mapsto c(rt)e^{c(rt)}$ 
	and thus $r\mapsto \partial_rc(r)$ are strictly monotonically increasing. This was 
	the assertion.
\end{proof}

With this we can prove the second of our main theorems:
\begin{proof}[Proof of Theorem \ref{thm:main2}]
 Combining Lemma \ref{lem.monotone} with Lemma \ref{lem.convex}, we have that $c$ is strictly convex. Moreover, as the exponential function is strictly convex and increasing,  $n=\alpha e^c$ is also strictly convex.
\end{proof}

Remaining in this setting, let us derive some estimates for $c$. It may be of particular interest to note that upon the choice of $r=R$ the following proposition also shows that in either of the limits $g\to \infty$ or $m\to 0$, the boundary condition turns into a Dirichlet boundary condition, as used in \cite{tuval}.

\begin{proposition}\label{boundaryproposition}
	Under the assumptions of Theorem \ref{thm:main2} it holds that
	\begin{equation}\label{estimate.for.boundary}
	\frac{g}{g+ \sqrt{\f m{|\Om|} e^{γ}}}e^{(r-R)\sqrt{\f m{|\Om|} e^{γ}} }\gamma\leq c(r) \leq \gamma. 
	\end{equation}
	for $0\leq r\leq R$.
\end{proposition}
\begin{proof}
First, not unlike \eqref{equation-radial}, we rewrite the equation for $c$ in spherical coordinates to
\begin{equation}\label{equation-radial2}
\partial_{rr} c+\frac{N-1}{r}\partial_r c=\frac{1}{r^{N-1}}\partial_r\left(r^{N-1}\partial_r c\right)=\alpha ce^c=nc, \end{equation}
for some $\alpha>0$ as in \eqref{relationbetweennandc}. 
 According to Lemma \ref{lem.monotone},  $\partial_r c\geq0$.
 In particular, this makes the second summand on the left of \eqref{equation-radial2} unnecessary and, moreover, shows that $n$ is maximal at $R$, so that
\[
\partial_{rr} c\leq n(R) c  \;\;\;\text{in } (0,R),\qquad \partial_r c(0)=0,\qquad \partial_r c(R)=gγ-gc(R). 
\]
Multiplying $\partial_{rr} c\le n(R)c$  
by $\partial_r c$, we obtain $\partial_r((\partial_r c)^2)\leq n(R) \partial_r(c^2)$ and thus 
\[
(\partial_r c)^2 \leq n(R) c^2\qquad\text{in } (0,R], 
\]
using that $c\geq0$. As $\partial_r c$ is non-negative in $(0,R)$ and $c(0)\geq0$,
\begin{equation}\label{radial-oneD-cprime}
\partial_r c \leq \sqrt{n(R)}c \qquad \text{in } (0,R].
\end{equation}
 Using Grönwall's inequality, we obtain that
\begin{equation}\label{cR-und-c0}
c(R)\leq c(r) e^{\sqrt{n(R)} (R-r)}
\end{equation}
for every $0<r\leq R$. We observe that the boundary condition and \eqref{radial-oneD-cprime} enable us to estimate $c(R)$ by means of 
\[
0\leq gγ-gc(R)=∂_r c(R) \leq \sqrt{n(R)} c(R),
\]
which is equivalent to the first inequality in
\begin{equation}\label{c.at.R}
\frac{g}{g+ \sqrt{n(R)}}\gamma\leq c(R)\leq γ,
\end{equation}
whereas the second results from $c(R)\le γ$, a consequence of nonnegativity of $∂_rc$ and the boundary condition.
Combining \eqref{c.at.R} with \eqref{cR-und-c0}, we furthermore obtain that
\begin{equation}\label{bondaryestimate-almostdone}
\frac{g}{g+ \sqrt{n(R)}}e^{\sqrt{n(R)} (r-R)}\gamma\leq c(r) \leq \gamma\qquad \text{for all } r\in(0,R].
\end{equation}
If we use that $α|\Om|\le α\io e^c = m$ and account for \eqref{c.at.R}, we can conclude that 
\[
 n(R)=αe^{c(R)}\le \f{m}{|\Om|}e^{γ}, 
\]
so that \eqref{estimate.for.boundary} follows from \eqref{bondaryestimate-almostdone}.
\end{proof}

\subsection{The one-dimensional case}\label{sec:one-D}

In this section, let us consider the one-dimensional setting, that is, $\Om$ being an interval. If $(n,c)$ solves \eqref{system}, from the previous sections we know that $n=αe^c$ (where $α$ depends monotonically on the total mass $\io n$ of bacteria), and $c$ solves $c'' = αce^c$. Apparently, $c$ is strictly convex, thus having precisely one minimum in $\Om$. Without loss of generality, we may assume that this is the case at $0$ and consider the problem in $(0,L)$, posing a homogeneous Neumann boundary condition at $0$ and the original boundary condition of \eqref{system} at $L$. With $G:=g(L)$, we hence are dealing with 
\[
c''=αce^c \;\text{in } (0,L),\qquad c'(0)=0,\qquad c'(L)=Gγ-Gc(L). 
\]
Multiplying $c''=αce^c$ by $c'$, we obtain $((c')^2)'=(2α(c-1)e^c)'$ and thus 
\[
(c')^2 = 2α(c-1)e^c -2α(c_0-1)e^{c_0}, 
\]
where $c_0:=\inf c = c(0)$. 
Due to $c'(0)=0$ and $c''$ being positive, we know that also $c'$ is positive in $(0,L)$. Therefore 
\begin{equation}\label{oneD-cprime}
c' = \sqrt{2α} \sqrt{(c-1)e^c - (c_0-1)e^{c_0}}
\end{equation}
and for every $x\in[0,L]$ we obtain 
\begin{equation}\label{oneD-implicitrepresentation}
\f1{\sqrt{2α}} \int_{c_0}^{c(x)} \f1{\sqrt{(c-1)e^c - (c_0-1)e^{c_0}}} dc = x.  
\end{equation}

In order to eliminate the unknown parameter $c_0$, we observe that the boundary condition and \eqref{oneD-cprime} enable us to express $c_0$ in terms of $c(L)$ by means of 
\[
Gγ-Gc(L) = \sqrt{2α} \sqrt{(c(L)-1)e^{c(L)} - (c_0-1)e^{c_0}}, 
\]
and that, thereby, finally, $c(L)$ can be obtained from \eqref{oneD-implicitrepresentation} with $x=L$. 

\subsection{Numerics} \label{sec:numerics}
In this section we show numerical solution of the system \eqref{system} in three dimensional domains. All numerical examples were implemented within the finite element library NGSolve/Netgen, see \cite{netgen,schoeberl2014cpp11}\footnote{The authors like to thank  Matthias Hochsteger, Lukas Kogler, Philip Lederer and Christoph Wintersteiger for their support using the NGSolve/Netgen library.}.

\subsubsection*{Example 1:} Let $\Omega=B_1(0)\subset \mathbb R^3$, $\gamma=g=1$ and $m=10\cdot |\Omega|=\frac{40}{3}\pi$. 
According to Theorem \ref{thm:main2}, there exists a unique classical solution $(n,c)$ of the system \eqref{system}. The uniqueness directly implies that $n$ and $c$ are radially symmetric. This can be observed on the cross section $\{x\in B_1(0): x_1=0\}$ in Figure \ref{fig2} and \ref{fig3}. Figure \ref{fig1} shows the dependency of the bacteria density $n$ and oxygen concentration $c$ on the radius $|x|$. The plots confirm that $n$ and $c$ are convex as proved by Theorem \ref{thm:main2}. Moreover, we see that $n$ is one magnitude larger than $c$. This can be explained as follows. By Lemma \ref{unique-up-to-multiple} and Lemma \ref{alpha-to-mass-injective}, we have that 
$n=\alpha e^c$ for some $\alpha>0$, where $\alpha$ is uniquely determined by $m$. In particular, 
for every $x,y\in \Omega$,
\[
n(x)\ge \frac{n(y)}{\max n} \cdot \min n.
\]
Hence, for all $x\in \Omega$,
\begin{align*}
n(x)\ge \frac{\min n}{\max n} \cdot \frac1{|B_1(0)|}\int_{B_1(0)} n(y) dy = e^{\min c - \max c} \frac{m}{|B_1(0)|}. 
\end{align*}
The numerical solution $c$ is bounded from below by $0.088$ and from above by $0.311$. Inserting these bounds together with $m/|B_1(0)|=10$, we obtain that $n(x)\geq 8$, which we can also observe in Figure \ref{fig1}.

\begin{center}
	\begin{figure}[!ht]
		\includegraphics[width=\linewidth]{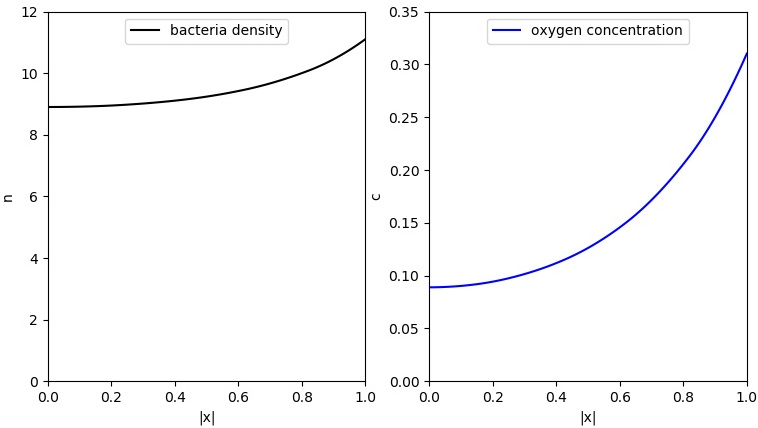}
		\caption{Dependency of the bacteria density $n$ and oxygen concentration $c$ on the radius in Example 1}
		\label{fig1}	
\end{figure}
\end{center}
\begin{center}
	\begin{figure}[!ht]
	\includegraphics[width=\linewidth]{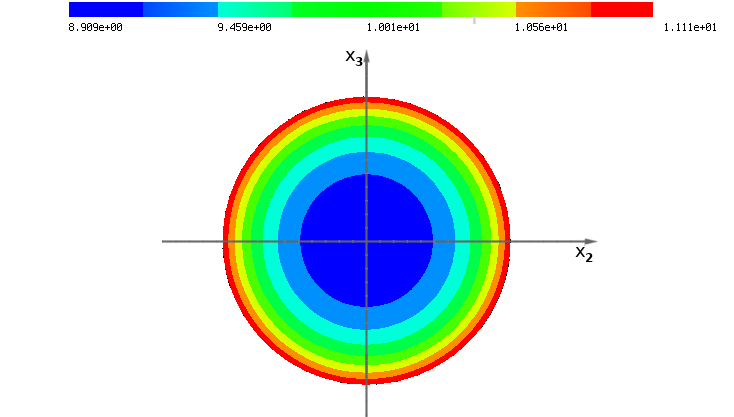}
	\caption{Bacteria density $n$ visualized on the cross section $\{x\in B_1(0): x_1=0\}$, Example 1}
	\label{fig2}	
	\end{figure}

\begin{figure}[!ht]
	\includegraphics[width=\linewidth]{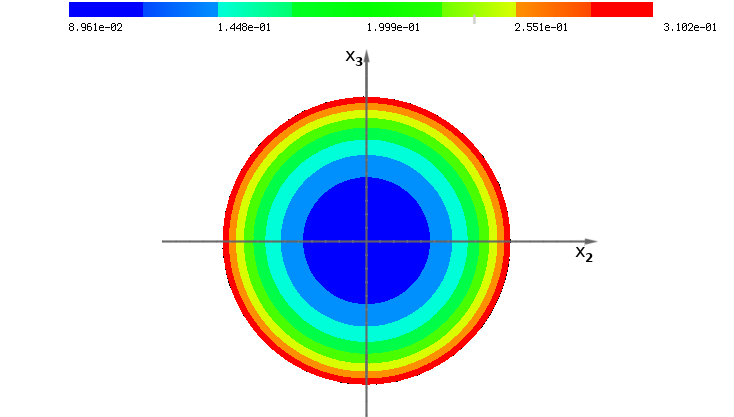}
 \caption{Oxygen concentration $c$ visualized on the cross section $\{x\in B_1(0): x_1=0\}$, Example 1}
 \label{fig3}	
	\end{figure}
	
\end{center}
\newpage
\subsubsection*{Example 2:} Let $\Omega=\{x\in B_1(0): x_3>-0.5\}$ and $\gamma=1$ and $g(x)=1 $ if $|x|=1$ and $g(x)=0$ else. The flat boundary of $\Omega$ shall model the boundary between water and a solid surface. We consider the solution with mass $\int_{\Omega}n=10|\Omega|$. Note that neither $\Omega$ nor $g$ are covered by the analytical results in this paper. Smooth approximations thereof, however, are. Assuming uniqueness of the solution $(n,c)$ yields that the solution is symmetric w.r.t.\ the axis $\{(0,0,z):z\in\mathbb R\}$. Therefore, the cross section $\{x\in\Omega: x_1=0\}$ as shown in Figure \ref{fig4} and \ref{fig5} contains all the information. We can see in Figure \ref{fig4} and \ref{fig5} that the bacteria density and the concentration of the oxygen have their largest value at the interface between water and gas. For the oxygen concentration, this is reasonable, because this interface acts as an oxygen source. As the bacteria prefer regions of higher oxygen 
concentration, they tend to move to this part of the boundary. Similarly to the previous example the difference the bacteria density is one magnitude larger than the oxygen concentration.
\begin{center}
	
	\begin{figure}[!ht]
	\includegraphics[width=\linewidth]{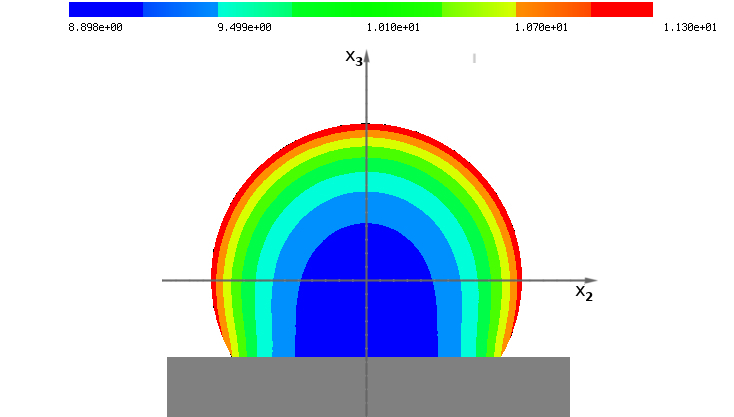}
	\caption{Bacteria density $n$ visualized on the cross section $\{x\in \Omega: x_1=0\}$, Example 2}
	\label{fig4}
\end{figure}
\begin{figure}[!ht]
	\includegraphics[width=\linewidth]{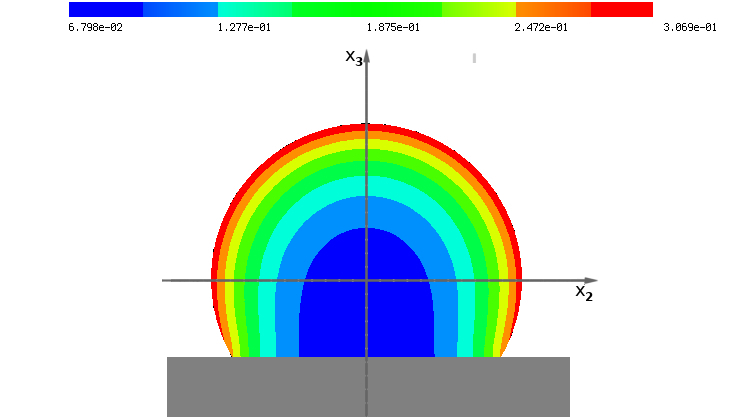}
	\caption{Oxygen concentration $c$ visualized on the cross section $\{x\in \Omega: x_1=0\}$, Example 2}
	\label{fig5}
\end{figure}

\end{center}

\newpage
\subsubsection*{Example 3:} 
In the third example, we consider again $\Omega=\{x\in B_1(0): x_3>-0.5\}$ and $\gamma=1$ and $g(x)=1 $ if $|x|=1$ and $g(x)=0$ else.
The analytical results in this article have already demonstrated that high concentrations near the boundary can largely be explained by the influence of the boundary condition. The experimental setting and the model in \cite{tuval} additionally involved interaction with the surrounding fluid. In this simulation, let us hence couple the equation to the stationary Navier-Stokes equation modeling the flow of water inside the drop as in \cite{tuval}, i.e.
\begin{equation*}
\begin{cases}
u\cdot \nabla n= \Delta n - \nabla\cdot(n\nabla c),\\
u \cdot \nabla c= \Delta c - nc,\\
u\nabla u = \Delta u +\nabla P - n\nabla \Phi, \qquad \nabla \cdot u= 0.
\end{cases}
\end{equation*}
We choose a relatively strong gravitational potential $\Phi(x_1,x_2,x_3)=100x_3$ in order to see the difference to system \eqref{system}. The boundary conditions are given by 
\begin{equation*}
\begin{aligned}
 \delny c &= (1-c)g&\text{on } \dOm\\
\delny n &= n\delny c&\text{on } \dOm\\
u &= 0&\text{on } \dOm.
\end{aligned}
\end{equation*}
Thus, this system is the stationary problem corresponding to the system \eqref{eq:ctns}.
The following plots show the numerical solution for given mass $\io n=10|\Omega|$.

\begin{center}
	
	\begin{figure}[!ht]
	\includegraphics[width=\linewidth]{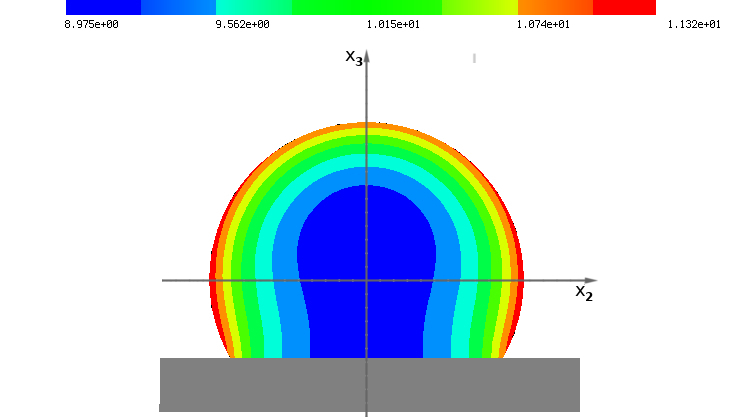}
	\caption{Bacteria density $n$ visualized on the cross section $\{x\in \Omega: x_1=0\}$, Example 3}
	\label{fig6}
\end{figure}
\begin{figure}[!ht]
	\includegraphics[width=\linewidth]{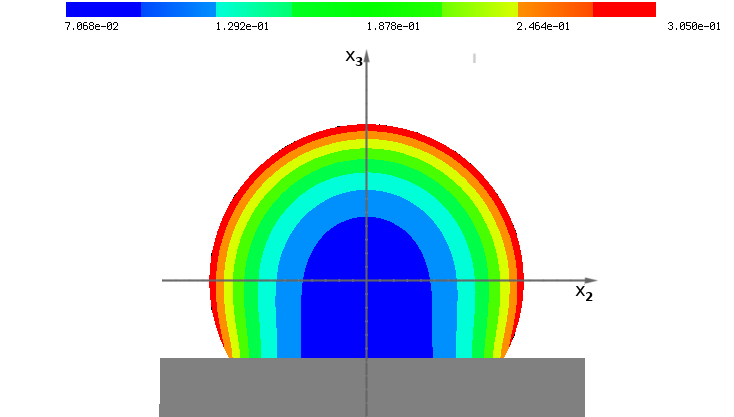}
	\caption{Oxygen concentration $c$ visualized on the cross section $\{x\in \Omega: x_1=0\}$, Example 3}
	\label{fig7}
\end{figure}
\end{center}

We again visualize $n$, $c$, $u$ and $P$ on the cross section $\{x\in\Omega: x_1=0\}$. The plot of the oxygen concentration $c$ in Figure \ref{fig7} is similar to the previous example. This can again be explained by the water-gas interface acting as an oxygen source. However, the bacteria density visualized in Figure \ref{fig6} is different to Figure \ref{fig4} because of the gravitation in the downward direction $\nabla\Phi=(0,0,-100)$. Therefore, the bacteria density at the top of the water drop is smaller than on the sides. Note that again the maximal bacteria density is to be expected at the water-gas interface because of the preference of the bacteria to higher oxygen concentration, which can be seen in Figure \ref{fig7}.

	
In Figure \ref{fig8}, we see that the flow $u$ is in downward direction at the sides of the water drop, where the bacteria density reaches its maximum. Inside the water drop, where $n$ takes its smallest value, the flow is directed into the opposite direction. The reason for this relies on the modeling assumptions that the flow is generated by the gravitational force of the bacteria. This moreover entails that the pressure $P$ admits its maximum at the bottom of the drop (see Figure \ref{fig9}).

\begin{center}
	\begin{figure}[!ht]
	\includegraphics[width=\linewidth]{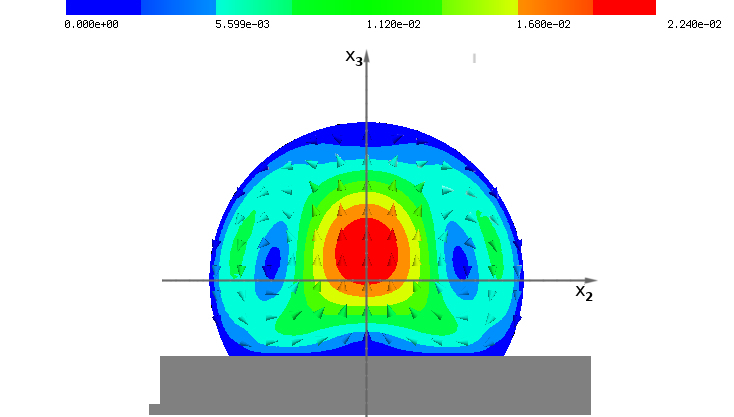}
	\caption{Flow $u$ visualized on the cross section $\{x\in \Omega: x_1=0\}$ - the colors reflect the magnitude $|u|$ and the arrows show the direction $\frac{u}{|u|}$, Example 3}
	\label{fig8}
\end{figure}

\end{center}
\begin{center}
	\begin{figure}[!ht]
	\includegraphics[width=\linewidth]{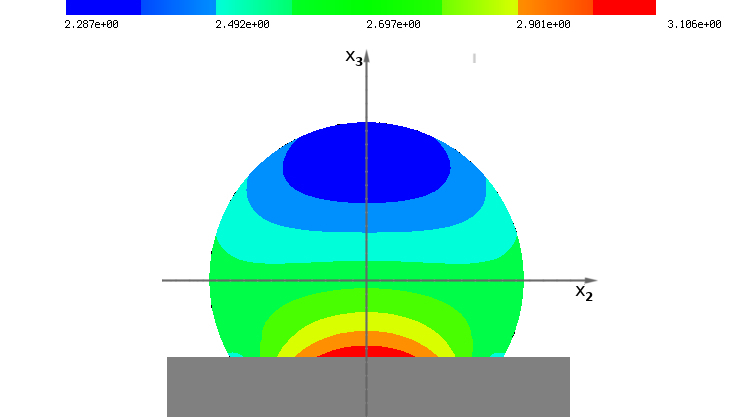}
	\caption{Pressure $p$ visualized on the cross section $\{x\in \Omega: x_1=0\}$, Example 3}
	\label{fig9}
\end{figure}

\end{center}
\newpage 
\section*{Acknowledgements}
The first author was funded by the Austrian Science Fund (FWF) project F 65. The second author acknowledges support of the {\em Deutsche Forschungsgemeinschaft} within the project {\em Analysis of chemotactic cross-diffusion in complex frameworks}.

{\footnotesize
\def\cprime{$'$}

}

\end{document}